\documentclass[12pt,a4paper]{article}
\usepackage{amsmath}
\usepackage{amsfonts}
\usepackage{amsthm}
\usepackage{amssymb}

\usepackage{enumerate}
\theoremstyle{plain}
\newtheorem{theo}{Theorem}[section]
\newtheorem{prop}[theo]{Proposition}
\theoremstyle{definition}
\newtheorem{definition}[theo]{Definition}
\theoremstyle{remark}
\newtheorem{rem}[theo]{Remark}
\numberwithin{equation}{section}

\newcommand{\R}{\mathbb{R}}

\newcommand{\divrg}{\textrm{div}\,}

\title{On doubling inequalities for elliptic systems
\thanks{Work supported in part by MiUR, PRIN no. 2006014115.}}
\author{Giovanni Alessandrini\thanks{Dipartimento di Matematica e Informatica, Universit\`a degli Studi di Trieste, via Valerio 12/1,
34127 Trieste, Italy.  \textsf{alessang@units.it}.}, Antonino
Morassi\thanks{Dipartimento di Georisorse e Territorio,
Universit\`a degli Studi di Udine, via Cotonificio 114, 33100
Udine, Italy. \textsf{antonino.morassi@uniud.it}.}, \\  Edi
Rosset\thanks{Dipartimento di Matematica e Informatica,
Universit\`a degli Studi di Trieste, via Valerio 12/1, 34127
Trieste, Italy.  \textsf{rossedi@units.it}.} \ and Sergio
Vessella\thanks{DIMAD, Universit\`a degli Studi di Firenze, Via
Lombroso 6/17, 50134 Firenze, Italy,
\textsf{sergio.vessella@dmd.unifi.it}.}}

\begin{document}

\maketitle

\noindent \textbf{Abstract.} We prove doubling inequalities for
solutions of elliptic systems with an iterated Laplacian as
diagonal principal part and for solutions of the Lam\'e system of
isotropic linearized elasticity. These inequalities depend on
global properties of the solutions.
\medskip

\medskip


\medskip

\medskip


\section{Introduction} \label{sec:introduction}
The doubling property is a basic measure theoretic concept
(\cite{l:gcrdf}, \cite{l:h}). Its connection with the strong
unique continuation principle for elliptic partial differential
equations became evident in the geometrical approach to unique
continuation developed by N. Garofalo and F.--H. Lin \cite{l:gl1,
l:gl2}. Subsequently, it turned out to be an important tool for
obtaining quantitative estimates suitable for stability estimates
in inverse boundary value problems \cite{abrvpisa, l:ve08, l:mr04,
l:mr05} and, also in connection with inverse boundary value
problems, for volume bounds of unknown inclusions in terms of
boundary measurements \cite{l:ar98, l:ars00, l:amr03}. Let us
illustrate the underlying idea with an example. Consider an
elliptic equation
\begin{equation}\label{eq:elliptic}
\divrg ( \sigma \nabla u ) = 0 \ , \text{ in } \Omega \ ,
\end{equation}
where $\Omega$ is a bounded open set with sufficiently smooth
boundary and $\sigma(x) = \{\sigma_{ij}(x)\}$ is a symmetric
matrix of coefficients, satisfying a uniform ellipticity condition
and such that $\sigma_{ij} \in C^{0,1}$, and consider the solution
$u$ to \eqref{eq:elliptic} satisfying the Dirichlet condition
\begin{equation}\label{eq:dirchcond}
 u  = g \ , \text{ on } \partial\Omega \ .
\end{equation}
The doubling property then says that for any compact subset $G$ of
$\Omega$ and for any concentric balls $B_r, B_{2r} \subset G$ we
have
\begin{equation}\label{eq:basicdoubl}
 \int_{B_{2r}}u^2 \leq K\int_{B_{r}}u^2 \ ,
\end{equation}
where $K$ is a constant which depends on $\Omega, G,$ the
ellipticity and regularity bounds on the coefficients, but also
\emph{necessarily}, on $u$. It is in fact evident, just by looking
at homogeneous harmonic polynomials in the unit ball, that the
above constant must diverge with the degree of the polynomials.

For the purposes of inverse boundary value problems, it is often
important that such a constant $K$ is estimated in terms of the
known boundary data $g$ and not on interior values of the solution
$u$ which may be unknown. Typically one expects that $C$ can be
bounded in terms of a ratio of the form
\begin{equation}\label{eq:basicfreq}
F(g) = \frac{\|g\|_{H^{1/2}(\partial
\Omega)}}{\|g\|_{L^{2}(\partial \Omega)}} \ .
\end{equation}
This ratio is usually called a \emph{frequency function}, and
Garofalo and Lin \cite{l:gl1} attributed this concept to Almgren
\cite{almg}. The specific choice of the norms in the ratio may
vary depending on the boundary value problem, and on the
functional framework. But the general idea is that the norm on the
numerator is of higher order than the one on the denominator so
that $F(g)$ resembles a Rayleigh quotient.

This theory can be considered well--settled within the area of
scalar elliptic equations \cite{l:ars00}. In the case of systems,
since the same issue of unique continuation maintains unanswered
questions, the study of doubling inequalities is still in
progress.

For the Lam\'e system of isotropic linearized elasticity, the
strong unique continuation is known when the coefficients $\mu,
\lambda \in C^{1,1}$ \cite{l:am01}, see also Escauriaza
\cite{esca}, for improved results in the $2$--dimensional case. In
fact, in \cite{l:am01} a doubling inequality of the following form
was proved
\begin{equation}\label{intdiv}
\int_{B_{2r}}|u|^2 + |\divrg u |^2\leq C\int_{B_{r}}|u|^2 +
|\divrg u |^2 \ ,
\end{equation}
from which the strong unique continuation can be easily derived.
However, it is not clear whether, from such an inequality, one can
derive a doubling inequality for $\int_{B_{r}}|u|^2$ only. In fact
such form of the doubling inequality was claimed in \cite[Theorem
$3.9$]{l:amr02}, but, unfortunately, the proof given there
contained a gap.

More recently, doubling inequalities have been studied for systems
with diagonal principal part which is either the Laplacian
$\Delta$ \cite{l:lnaka08} or the iterated Laplacian $\Delta^l$
\cite{l:lnw08}, and in fact the coefficients in the lower order
terms are also allowed to be singular. In these papers, local
forms of doubling inequalities were obtained.

In this note, our aim is twofold. First, we show that for elliptic
systems \emph{with diagonal principal part given by $\Delta^l$ and
bounded lower order terms}, a global form of doubling inequality
holds, see Theorem \ref{theo:Doubling-LNW}. Second, we apply this
result to the Lam\'e system, by observing that, assuming in
addition $\mu, \lambda \in C^{2,1}$, such a system can be reduced
to a $4$--th order system with $\Delta^2$ as its diagonal
principal part. Thus by such means, we restore the validity of the
claimed Theorem $3.9$ in \cite{l:amr02} and consequently of
Proposition $4.3$ in \cite{l:amr02} and Theorem 4.8 in
\cite{l:amr03}, at least under the regularity assumptions $\mu,
\lambda \in C^{2,1}$.

It remains open the issue whether the doubling inequality holds
under the assumption $\mu, \lambda \in C^{1,1}$, whereas it is
well--known that a challenging open question is whether unique
continuation in general holds true when $\mu, \lambda \in
C^{0,1}$, we refer again to Escauriaza \cite{esca}, for the state
of the art in the $2$--dimensional case.

The plan of the paper is as follows. In Section \ref{sec:notation}
we set up our notation and formulate the pure traction boundary
value problem for the Lam\'e system of linearized elasticity. In
Section \ref{sec:results} we first formulate a three--spheres
inequality for solutions of systems with $\Delta^l$ as diagonal
principal part, Theorem \ref{theo:3spheres-LNW}, which is an
immediate  consequence of a result  of C.--L. Lin, S. Nagayasu,
and J.N. Wang \cite{l:lnw08}. Next we apply such a three--spheres
inequality to derive a so--called estimate of \emph{propagation of
smallness}, Theorem \ref{theo:LipPropSmall}. Then, in Theorem
\ref{theo:localDoubl}, we recall the local version of the doubling
inequality proved by C.--L. Lin, S. Nagayasu, and J.N. Wang
\cite[Theorem $1.3$]{l:lnw08}, and we arrive at our global
version, Theorem \ref{theo:Doubling-LNW}. The doubling inequality
we obtain has a constant $K$ which, among other quantities,
depends on a frequency function given by the ratio
$\|u\|_{H^{1/2}(\Omega)}/\|u\|_{L^2(\Omega)}$. Depending on which
is the appropriate boundary condition that may be prescribed, such
ratio could be dominated by a suitable ratio of norms which only
involve the boundary data. This process is exemplified in the
following Theorem \ref{prop:Doubling_global} where the doubling
inequality for the Lam\'e system is obtained and the doubling
constant $K$ is controlled in terms of a ratio of norms of the
boundary traction field $\varphi$. The bridge between the two
Theorems \ref{theo:Doubling-LNW}, \ref{prop:Doubling_global} is
provided by Proposition \ref{prop:Reduction} which enables to
reduce the Lam\'e system to a system with $\Delta^2$ as diagonal
principal part.

\section{Notation} \label{sec:notation}

Throughout this paper we shall consider a bounded domain $\Omega$
in ${\R}^{n}$, $n\geq2$, having Lipschitz boundary with constants
$r_{0}$, $M_{0}$ according to the following definition.
\begin{definition}
  \label{def:Lip} (Lipschitz regularity)
  Given a domain $\Omega$, we shall say that $\partial \Omega$ is of
\textit{Lipschitz class with constants
  $r_0$, $M_0$}, if, for any $x_{0} \in \partial \Omega$,
  there exists a rigid transformation of coordinates under which we
  have $x_{0}=0$ and
\begin{equation*}
  \Omega \cap B_{r_0}(0)=\{x \in B_{r_0}(0)\quad | \quad
  x_{n}>\psi(x')
  \},
\end{equation*}
where for $x\in\R^n$, we set $x=(x',x_n)$, with $x'\in R^{n-1}$,
$x_n\in \R$ and where $\psi$ is a Lipschitz continuous function on
$B_{r_0}(0) \subset {\R}^{n-1}$ satisfying
\begin{equation*}
\psi(0)=0
\end{equation*}
and
\begin{equation*}
\|\psi\|_{{C}^{0,1}(B_{r_0}(0))} \leq M_0r_0.
\end{equation*}
\end{definition}
Given a bounded domain $\Omega \subset {\R}^{n}$, for any $d>0$ we
shall denote
\begin{equation}
  \label{eq:2.1}
  \Omega_{d}=\{x \in \Omega \mid \textrm{dist}(x,\partial \Omega)>d \}.
\end{equation}

Moreover, when no ambiguity occurs, we shall denote for brevity by
$B_R$ any ball in $\R^n$ of radius $R$.

Let us consider weak solutions $u \in H^1(\Omega, \R^n)$ to the
Lam\'e system
\begin{equation}
  \label{eq:sys-Lame}
  \divrg (\mu (\nabla u + (\nabla u)^T)) + \nabla(\lambda \divrg u)=0
  \quad \mathrm{in}\  \Omega,
\end{equation}
which describes the equilibrium of a body $\Omega$ made by linear
elastic isotropic material when body forces are absent. Here,
$(\nabla u)^T$ denotes the transpose of the matrix $\nabla u$. In
equation \eqref{eq:sys-Lame}, $\mu=\mu(x)$ and
$\lambda=\lambda(x)$ are the Lam\'e moduli of the material.

In this paper we shall assume $\mu \in
C^{2,1}(\overline{\Omega})$, $\lambda \in
C^{2,1}(\overline{\Omega})$ with
\begin{equation}
    \label{eq:reg-bound}
    \|\mu\|_{C^{2,1}(\overline{\Omega})}+\|\lambda\|_{C^{2,1}(\overline{\Omega})}
    \leq M,
\end{equation}
for some positive constant $M$.

We shall say that $\mu$ and $\lambda$ satisfy the \textit{strong
convexity} condition if
\begin{equation}
    \label{eq:strong-convex}
    \mu(x) \geq \alpha_0>0, \quad 2\mu(x)+n\lambda(x) \geq
    \gamma_0 >0 \quad \hbox{in } \overline{\Omega},
\end{equation}
whereas they satisfy the \textit{strong ellipticity} condition if
\begin{equation}
    \label{eq:strong-ellipt}
    \mu(x) \geq \alpha_0>0, \quad 2\mu(x)+\lambda(x) \geq
    \beta_0 >0 \quad \hbox{in } \overline{\Omega},
\end{equation}
where $\alpha_0$, $\beta_0$, $\gamma_0$ are positive constants. It
is well known that condition \eqref{eq:strong-convex} implies
\eqref{eq:strong-ellipt}, with
$\beta_0=\min\{2\alpha_0,\gamma_0\}$.

We shall prescribe a boundary traction field $\varphi \in
{L}^{2}(\partial \Omega, {\R}^{n})$ satisfying the compatibility
condition
\begin{equation}
  \label{eq:2.21}
  \int_{\partial \Omega} \varphi \cdot r =0
  \end{equation}
for every \textit{infinitesimal rigid displacement} $r$, that is
$r(x)=c+Wx$, where $c$ any constant $n-$vector and $W$ is any
constant skew $n \times n$ matrix. Namely, we shall consider weak
solutions $u \in H^{1}(\Omega,\R^{n})$ of the following problem:
\begin{equation}
  \label{eq:2.22}
  \divrg (\mu (\nabla u +(\nabla u)^{T}))+\nabla(\lambda \divrg u)=0 \quad \mathrm{in}\ \Omega,
\end{equation}
\begin{equation}
  \label{eq:2.23}
  (\mu (\nabla u +(\nabla u)^{T}))+ \lambda(\divrg u)I_n) \nu=\varphi \quad \mathrm{on}\ \partial
  \Omega,
\end{equation}
where $I_n$ is the $n \times n$ identity matrix and $\nu$ is the
unit exterior normal to $\partial \Omega$.

Regarding existence, we recall that, provided the compatibility
condition  \eqref{eq:2.21} is satisfied, a solution of the
traction problem \eqref{eq:2.22}, \eqref{eq:2.23} exists as long
as the Lam\'e moduli $\mu$ and $\lambda$ either belong to
$L^\infty({\Omega})$ and satisfy the strong convexity condition,
or they are continuous and satisfy the strong ellipticity
condition, see for instance Valent \cite[\S III]{l:v88}.

With respect to uniqueness, it is well-known that the solution $u$
to the above problem is uniquely determined up to an infinitesimal
rigid displacement. In order to uniquely identify such solution,
we shall assume {from} now on that $u$ satisfies the following
normalization conditions
\begin{equation}
  \label{eq:2.25bis}
  \int_{\Omega} u =0, \quad \int_{\Omega} (\nabla u - (\nabla u)^{T})=0.
  \end{equation}

\section{Results} \label{sec:results}

Let $\Omega$ be a bounded domain in $\R^n$ with boundary $\partial
\Omega$ of Lipschitz class with constants $r_0$ and $M_0$. Let
$u=(u^1, ..., u^n) \in H^{2l}(\Omega, \R^n)$ be a solution to the
system of differential inequalities
\begin{equation}
  \label{eq:p1-e2}
  |\Delta^l u^i| \leq K_0 \sum_{|\alpha|\leq \left [\frac{3l}{2} \right ]}
  |D^{\alpha}u| \ ,
  \ \ \ i=1,...,n\ .
\end{equation}

\begin{theo} [Three spheres inequality]
  \label{theo:3spheres-LNW}
    Let $B_R \subset \Omega$. There exists a positive
    number $\vartheta < e^{-1/2}$, only depending on $n$, $l$, $K_0$,
     such that for every $r_1, r_2,
    r_3$, $0<r_1<r_2<\vartheta r_3$, $r_3 \leq R$, we have
\begin{equation}
  \label{eq:p1-e3}
  \int_{B_{r_2}} |u|^2 dx \leq C
  \left  ( \int_{B_{r_1}} |u|^2 dx \right )^\delta
  \left  ( \int_{B_{r_3}} |u|^2 dx \right )^{1-\delta},
\end{equation}
for every $u \in H^{2l}(\Omega, \R^n)$ satisfying
\eqref{eq:p1-e2}, where the constants $C$ and $\delta$, $C>0$,
$0<\delta < 1$, only depend on $n$, $l$, $K_0$, $r_1/r_3$,
$r_2/r_3$, and where the balls $B_{r_i}$, $i=1,2,3$, have the same
center as $B_R$.
\end{theo}
\begin{proof} This is in fact a special case of Theorem $1.1$ in \cite{l:lnw08},
where, in the inequalities \eqref{eq:p1-e2}, suitable
singularities at the center of the balls $B_{r_i}$,  $B_R$ are
also allowed.
\end{proof}

\begin{theo} [Lipschitz propagation of smallness]
  \label{theo:LipPropSmall}
  Under the previous assumptions, for every $\rho >0$ and for
  every $x \in \Omega_{\frac{4\rho}{\vartheta}}$, we have
\begin{equation}
  \label{eq:p2-e4}
  \int_{B_{\rho}(x)} |u|^2 dx \geq C_{\rho}
  \int_{\Omega} |u|^2 dx,
\end{equation}
where $\vartheta$ has been defined in Theorem
\ref{theo:3spheres-LNW} and $C_\rho$ only depends on $n$, $l$,
$K_0$, $r_0$, $M_0$, $|\Omega|$,
$\|u\|_{H^{1/2}(\Omega)}/\|u\|_{L^2(\Omega)}$ and $\rho$.
\end{theo}

\begin{proof}
By an iterative application of the three spheres inequality
\eqref{eq:p1-e3} over balls having fixed values of the ratio
$r_1/r_3$, $r_2/r_3$, and by repeating the arguments in
\cite[Proposition $4.1$]{l:amr02}  we have
\begin{equation}
    \label{eq:p2bis-eA}
    \frac{\|u\|_{L^2 (\Omega_{\frac{5\rho} {\theta}}  )}}
    {\|u\|_{L^2(\Omega)}}
    \leqslant
    \frac{C}{\rho^{n/2}}
    \left( \frac{\|u\|_{L^2(B_{\rho}(x))}}
    {\|u\|_{L^2(\Omega)}}
    \right) ^{\delta^{L}},
\end{equation}
with $L\leqslant \frac{|\Omega|} {\omega_n\rho^n}$. Here, the
constants $C>0$ and $\delta$, $0<\delta<1$, only depend on $n, l,
K_0$.

We can rewrite the square of the left-hand side of
\eqref{eq:p2bis-eA} as
\begin{equation}
    \label{eq:p2bis-eB}
    \frac{\|u\|_{L^2  (\Omega_{\frac{5\rho} {\theta}}  )}^2}
    {\|u\|_{L^2(\Omega)}^2}
    =
    1-
    \frac{\|u\|_{L^2 (\Omega \setminus \Omega_{\frac{5\rho} {\theta}}  )}^2}
    {\|u\|_{L^2(\Omega)}^2}.
\end{equation}
By H\"{o}lder's inequality
\begin{equation}
    \label{eq:p2bis-eC}
    \|u\|^2_{L^2(\Omega\setminus\Omega_ {\frac{5\rho} {{\theta}}})}
    \leqslant
    \left|\Omega\setminus \Omega_{\frac{5\rho} {{\theta}}}\right|^{1/n} \|u\|^2_{L^{2n/(n-1)}
     (\Omega\setminus\Omega_{\frac{5\rho}
    {{\theta}}}  )}
\end{equation}
and by Sobolev inequality (see, for instance, \cite{l:ad75})
\begin{equation}
    \label{eq:p2bis-eD}
    \|u\|^2_{L^{2n/(n-1)}(\Omega)}\leqslant C\|u\|^2_{H^{1/2}(\Omega)},
\end{equation}
we have
\begin{equation}
    \label{eq:p2ter-eE}
    \|u\|^2_{L^2  (\Omega\setminus\Omega_ {\frac{5\rho}{{\theta}}}  )}
    \leqslant
    C\left |\Omega\setminus \Omega_{\frac{5\rho}{{\theta}}} \right |^{1/n} \|u\| ^2_{H^{1/2}(\Omega)},
\end{equation}
where $C>0$ only depends on $r_0$, $M_0$ and $|\Omega|$.

Moreover,
\begin{equation}
    \label{eq:p2ter-eF}
    \left |\Omega\setminus \Omega_{\frac{5\rho} {{\theta}}} \right |\leqslant C\rho,
\end{equation}
where $C>0$ only depends on $r_0$, $M_0$ and $|\Omega|$ (see
estimate $(A.3)$ in \cite{l:ar98} for details).

By \eqref{eq:p2bis-eB}, \eqref{eq:p2ter-eE} and
\eqref{eq:p2ter-eF} we have that there exists $\bar{\rho}>0$, only
depending on $r_0, M_0, |\Omega|$ and
$\|u\|_{H^{1/2}(\Omega)}/\|u\|_{L^{2}(\Omega)}$ such that
\begin{equation}
    \label{eq:p2ter-eG}
    \frac{\|u\|_{L^2 (\Omega_{\frac{5\rho} {{\theta}}} )}^2}
    {\|u\|_{L^2(\Omega)}^2}
    \geqslant
    \frac{1} {2},
\end{equation}
for every $\rho$, $0<\rho\leqslant \bar{\rho}$.

Therefore, {from} \eqref{eq:p2bis-eA} and \eqref{eq:p2ter-eG} the
thesis follows when $0<\rho\leqslant \overline{\rho}$.

For larger values of $\rho$, inequality \eqref{eq:p2-e4} is
trivial.
\end{proof}
\begin{theo}[Local doubling inequality]
  \label{theo:localDoubl}
    Let $u\in H^{2l}(B_1, \R^n)$ be a nontrivial solution to \eqref{eq:p1-e2}
    in $B_1\subset\R^n$.
    There exist constants  $R_0\in (0,1)$, $\vartheta^*\in (0,\frac{1}{2})$,
    $K>0$ such that
\begin{equation}
  \label{eq:p3-e4}
  \int_{B_{2r}} |u|^2 dx \leq K
  \int_{B_{r}} |u|^2 dx, \quad \hbox{for every } r,\ 0<r\leq\vartheta^*\ .
\end{equation}
Here $R_0$ only depends on $n, l, K_0$, whereas $\vartheta^*$, $K$
only depend on $n, l, K_0$ and on the ratio
\begin{equation}
  \label{eq:freqloc}
    F_{loc} = \frac{ \|u\|_{L^2(B_{R_0^2})} }{  \|u\|_{L^2(B_{R_0^4})}  }
   \ .
\end{equation}
\end{theo}
\begin{proof}
We refer to Theorem $1.3$ in \cite{l:lnw08}. The present statement
is merely adapted in terms of notation and of a more explicit
expression of the dependencies of the various constants $R_0$,
$\vartheta^*$, $K$.
\end{proof}
\begin{theo}[Doubling inequality]
  \label{theo:Doubling-LNW}
    Let $\Omega$ be a bounded domain in $\R^n$, $n \geq 2$,
    with boundary of Lipschitz class with constants $r_0$,
    $M_0$ and let $u\in H^{2l}(\Omega, \R^n)$ be a nontrivial solution
    to \eqref{eq:p1-e2}.
    There exists a constant $\vartheta$, $0<\vartheta < 1$,
    only depending on $n, l, K_0$, such that for every $\bar r>0$
    and for every $x_0\in\Omega_{\bar r}$, we have
\begin{equation}
  \label{eq:p3-e4_bis}
  \int_{B_{2r}(x_0)} |u|^2 dx \leq K
  \int_{B_{r}(x_0)} |u|^2 dx, \quad \hbox{for every } r,\ 0<r\leq \frac{\vartheta}{2}\bar
  r,
\end{equation}
where $K>0$ only depends on $n$, $l$, $K_0$, $r_0$, $M_0$,
$|\Omega|$, $\bar r$ and
$\|u\|_{H^{1/2}(\Omega)}/\|u\|_{L^2(\Omega)}$.
\end{theo}
\begin{proof}
By the unique continuation property, $u$ is a nontrivial solution
to \eqref{eq:p1-e2} in $B_{\bar r}(x_0)\subset\Omega$.

Let
\begin{equation*}
   v(y)=u(\bar ry+x_0).
\end{equation*}
Then $v\in H^{2l}(B_1,\R^n)$ is a nontrivial solution in $B_1$ to
\begin{equation}
  \label{eq:p1-e2_traslata}
  |\Delta^l v^i(y)| \leq \widetilde{K_0} \sum_{|\alpha|\leq \left [\frac{3l}{2} \right ]}
  |D^{\alpha}v(y)| \ ,
  \ \ \ i=1,...,n\ ,
\end{equation}
with $\widetilde{K_0}$ only depending on $n$, $l$, $K_0$, $\bar
r$.

By theorem \ref{theo:localDoubl} and coming back to the old
variables, we have
\begin{equation}
  \label{eq:doubling_piccoliraggi}
  \int_{B_{2s}(x_0)} |u|^2 dx \leq K
  \int_{B_{s}(x_0)} |u|^2 dx, \quad \hbox{for every } s,\ 0<s\leq \vartheta^*\bar
  r,
\end{equation}
with $\vartheta^*\in (0,\frac{1}{2})$, $K>0$ only depending on
$n$, $l$, $K_0$, $\bar r$ and, increasingly, on
\begin{equation}
  \label{eq:F_tilde}
    \widetilde{F}_{loc} = \frac{ \|u\|_{L^2(B_{R_0^2\bar r}(x_0))} }
    {  \|u\|_{L^2(B_{R_0^4\bar r}(x_0))}  }
\end{equation}
Let $\vartheta\in (0,1)$ be the constant introduced in Theorem
\ref{theo:3spheres-LNW}. If $\vartheta^*\geq\frac{\vartheta}{2}$,
then \eqref{eq:doubling_piccoliraggi} holds for
$s\leq\frac{\vartheta}{2}\bar r$. Otherwise, given $s\in
(\vartheta^*\bar r, \frac{\vartheta}{2}\bar r)$, by applying
Theorem \ref{theo:3spheres-LNW} with $r_1=\vartheta^*\bar r$,
$r_2=2s$, $r_3=\bar r$, we have
\begin{equation}
  \label{eq:doubling+3sfere1}
  \int_{B_{2s}(x_0)} |u|^2 dx \leq C\left(\int_{B_{\vartheta^*\bar r}(x_0)}
  |u|^2 dx\right)^\delta
  \left(\int_{B_{\bar r}(x_0)},
  |u|^2 dx\right)^{1-\delta}
\end{equation}
with $\delta\in(0,1)$, $C>0$ only depending on $n$, $l$, $K_0$,
$\frac{r_1}{r_3}=\vartheta^*$, $\frac{r_2}{r_3}=\frac{2s}{\bar
r}$. Let us notice that the constant $C$ depends increasingly on
$\frac{r_2}{r_3}=\frac{2s}{\bar r}<\vartheta$. Since $\vartheta$
only depends on $n$, $l$, $K_0$, we have that $C$ only depends on
$n$, $l$, $K_0$, $\vartheta^*$. We have
\begin{equation}
  \label{eq:doubling+3sfere2}
  \frac{\int_{B_{2s}(x_0)} |u|^2 dx}{\int_{B_{\vartheta^*\bar r}(x_0)}
  |u|^2 dx} \leq C\left(\frac{\int_{B_{\bar r}(x_0)}
  |u|^2 dx}{\int_{B_{\vartheta^*\bar r}(x_0)}
  |u|^2 dx}\right)^{1-\delta}\leq C\frac{\int_{B_{\bar r}(x_0)}
  |u|^2 dx}{\int_{B_{\vartheta^*\bar r}(x_0)}
  |u|^2 dx},
\end{equation}
and therefore, recalling that $\vartheta^*\bar r<s$, we have
\begin{equation}
  \label{eq:doubling+3sfere3}
  \int_{B_{2s}(x_0)} |u|^2 dx\leq C\frac{\|u\|^2_{L^2(B_{\bar r}(x_0))}}
  {\|u\|^2_{L^2(B_{\vartheta^*\bar r}(x_0))}}
  \int_{B_{s}(x_0)} |u|^2
  dx,
\end{equation}
for every $s$, $\vartheta^*\bar r<s<\frac{\vartheta}{2}\bar r$.

Let us estimate $\widetilde{F}_{loc}$. Let $\rho=\min\{R_0^4\bar
r,\frac{\vartheta}{4}\bar r\}$. By applying Theorem
\ref{theo:LipPropSmall}, we have
\begin{equation}
  \label{eq:stima_Floc}
    \widetilde{F}_{loc}
    \leq
    \frac{ \|u\|_{L^2(\Omega)} }{ \|u\|_{L^2(B_\rho(x_0))}
    }\leq \frac{1}{\sqrt{C_\rho}},
\end{equation}
with $C_\rho$ only depending on $n$, $l$, $K_0$, $r_0$, $M_0$,
$|\Omega|$, $\bar r$ and
$\|u\|_{H^{1/2}(\Omega)}/\|u\|_{L^2(\Omega)}$.

Therefore $\vartheta^*$ and the constant $C$ appearing in
\eqref{eq:doubling+3sfere3} only depend on the above constants.

We can now estimate $\|u\|_{L^2(B_{\bar
r}(x_0))}/\|u\|_{L^2(B_{\vartheta^*\bar r}(x_0))}$ with the same
quantities by applying analogously Theorem
\ref{theo:LipPropSmall}.

The thesis follows {}from \eqref{eq:doubling_piccoliraggi} and
\eqref{eq:doubling+3sfere3}.
\end{proof}

\begin{prop}
    \label{prop:Reduction}
    Let $\Omega$ be a bounded domain in $\R^n$, $n \geq 2$. Let
    the Lam\'e moduli $\mu, \lambda \in C^{2,1}(
    \overline{\Omega})$ satisfy the strong ellipticity conditions
\begin{equation}
    \label{eq:p4-e6}
    \mu(x) \geq \alpha_0>0, \quad 2\mu(x)+\lambda(x) \geq
    \beta_0 >0 \quad \hbox{for every } x\in \overline{\Omega}
\end{equation}
and the upper bound
\begin{equation}
    \label{eq:p4-e7}
    \|\mu\|_{C^{2,1}(\overline{\Omega})}+\|\lambda\|_{C^{2,1}(\overline{\Omega})}
    \leq M,
\end{equation}
where $\alpha_0$, $\beta_0$, $M$ are given positive constants.
Then, there exists a positive constant $K_0$ only depending on $n,
\alpha_0, \beta_0, M$ such that, for every solution $u \in
H_{loc}^4 (\Omega, \R^n)$ of the Lam\'e system
\begin{equation}
  \label{eq:p4-e9}
  \mathrm{div} (\mu (\nabla u + (\nabla u)^T)) + \nabla(\lambda \mathrm{div}
  u)=0,
  \quad \hbox{in}\  \Omega \ ,
\end{equation}
we also have
\begin{equation}
    \label{eq:p4-e8}
    |\Delta^2 u^i| \leq K_0 \sum_{|\alpha|=1}^3 \left | D^{\alpha} u \right |, \quad
    i=1,...,n\ .
\end{equation}
\end{prop}

\begin{rem}
  \label{rem:dipag4}
  Let us notice that, being $\mu, \lambda \in C^{2,1}(\overline{\Omega})$,
  by interior regularity estimates for the Lam\'e system, we
  have that for any weak solution $u$ to \eqref{eq:p4-e9} we also have $u \in H_{loc}^4(\Omega,\R^n)$. See, for instance,
  \cite{l:camp80}. Consequently \eqref{eq:p4-e8} is indeed valid
  for any weak solution to the Lam\'e system.
\end{rem}

\begin{proof}
In what follows we denote
\begin{equation}\label{vector}
\Pi = \left( \begin{array}{c} \mu \\ \lambda \end{array} \right) \
,
\end{equation}
and, for any function $v$ we denote by $D^k v$ the set of all
derivatives of order $k$ of $v$. Moreover, we shall denote by $B_j
(X;Y)$, $j=1,2,\ldots$ bilinear (vector valued) functions of the
vectors (or tensors) $X$ and $Y$, their explicit expression shall
vary from line to line.

We can rewrite \eqref{eq:p4-e9} as follows
\begin{equation}
    \label{eq:p5-e10}
    \mu \Delta u^j + (\mu+\lambda)(\divrg
    u)_{x_j}=B_j(D\Pi;Du), \quad j=1,...,n \ .
\end{equation}
Differentiating by $x_j$ and summing up, we have
\begin{equation}
    \label{eq:p5-e11}
    (2\mu+\lambda) \Delta (\divrg u) = B_1(D\Pi;D^2u)+B_2(D^2\Pi;Du) \ ,
\end{equation}
Differentiating once more into equation \eqref{eq:p5-e11}, we
obtain, in the almost everywhere sense,
\begin{equation}
    \label{eq:p6-e12}
    (2\mu+\lambda) \nabla(\Delta(\divrg u))
    = B_1(D\Pi;D^3u)+B_2(D^2\Pi;D^2u)+B_3(D^3\Pi;Du) \ .
\end{equation}
By applying the Laplacian to  \eqref{eq:p5-e10} we also have
\begin{equation}
    \label{eq:p6-e13}
    \mu \Delta^2 u + (\mu + \lambda) \nabla(\Delta(\divrg u))= B_1(D\Pi;D^3u)+B_2(D^2\Pi;D^2u)+B_3(D^3\Pi;Du) \
    ,
\end{equation}
in the almost everywhere sense.

With the aid of the strong ellipticity conditions \eqref{eq:p4-e6}
we can eliminate the term $\nabla(\Delta(\divrg u))$ from the
equations \eqref{eq:p6-e12} and \eqref{eq:p6-e13}. Recalling the
bounds \eqref{eq:p4-e7} we arrive at \eqref{eq:p4-e8}.
\end{proof}

\begin{theo} [Global doubling inequality]
    \label{prop:Doubling_global}
    Let $\Omega$ be a bounded domain in $\R^n$, $n \geq 2$, with boundary of
    Lipschitz class with constants $r_0$,
    $M_0$. Let $u \in H^1(\Omega, \R^n)$ be a weak solution to the
    boundary value problem \eqref{eq:2.22}, \eqref{eq:2.23} satisfying
    the normalization conditions \eqref{eq:2.25bis}.
    Let $\mu$, $\lambda \in C^{2,1}( \overline{\Omega})$
    satisfy the regularity condition \eqref{eq:p4-e7} and
    the strong convexity condition
    \eqref{eq:strong-convex}.

    There exists a constant $\vartheta$, $0<\vartheta < 1$,
    only depending on $n, \alpha_0, \gamma_0, M$, such that for every $\bar r>0$
    and for every $x_0\in\Omega_{\bar r}$, we have
\begin{equation}
  \label{eq:p3-e4_ter}
  \int_{B_{2r}} |u|^2 dx \leq K
  \int_{B_{r}} |u|^2 dx, \quad \hbox{for every } r, 0<r\leq \frac{\vartheta}{2}\bar
  r,
\end{equation}
where the constant $K>0$ only depends on $n$, $\alpha_0$,
$\gamma_0$, $M$, $r_0$, $M_0$, $|\Omega|$, $\bar r$ and $ \|
\varphi\|_{H^{-1/2}(\partial \Omega)}/\|
\varphi\|_{H^{-1}(\partial \Omega)}$.
\end{theo}

\begin{proof}
By applying Theorem \ref{theo:Doubling-LNW} and Proposition
\ref{prop:Reduction} we infer that \eqref{eq:p3-e4_ter} holds with
the constant $K$ only depending on $n, \alpha_0, \beta_0, M,
|\Omega|, r_0, M_0$ and
$\|u\|_{H^{1/2}(\Omega)}/\|u\|_{L^2(\Omega)}$. By the weak
formulation of the problem \eqref{eq:2.22}, \eqref{eq:2.23} and by
the normalization conditions \eqref{eq:2.25bis} we have
\begin{equation}
  \label{eq:p8-e16}
  \|u\|_{H^{1}(\Omega)} \leq C \|\varphi\|_{H^{-1/2}(\partial
  \Omega)},
\end{equation}
where $C>0$ only depends on $n, r_0, M_0, |\Omega|, \alpha_0,
\gamma_0$. Moreover the following interpolation inequality holds
\begin{equation}
  \label{eq:p8-e17}
  \|u\|_{H^{1/2}(\Omega)}^2 \leq
  \|u\|_{H^{1}(\Omega)}\|u\|_{L^{2}(\Omega)}.
\end{equation}
Let us now recall the trace inequality (see, for instance,
\cite[Theorem 1.5.1.10]{l:gri85})
\begin{equation}
  \label{eq:p8-e18}
  \|u\|_{L^{2}(\partial \Omega)}^2 \leq C
  \|u\|_{L^{2}(\Omega)}\|u\|_{H^{1}(\Omega)},
\end{equation}
where $C$ only depends on $r_0$, $M_0$, $|\Omega|$, and the
estimate of Lemma 4.10 in \cite{l:amr03}
\begin{equation}
  \label{eq:p8-e19}
  \|\varphi\|_{H^{-1}(\partial \Omega)}\leq C \|u\|_{L^{2}(\partial
  \Omega)},
\end{equation}
where $C>0$ only depends on $|\Omega|, r_0, M_0, \alpha_0,
\gamma_0$ and $M$.

Therefore
\begin{equation}
    \label{eq:p9-e20}
    \frac{\|u\|_{H^{1/2}(\Omega)}^2  }{\|u\|_{L^{2}(\Omega)}^2}
    \leq
    \frac{\|u\|_{H^{1}(\Omega)}  }{\|u\|_{L^{2}(\Omega)}}
    \leq
    C \frac{\|u\|_{H^{1}(\Omega)}^2  }{\|u\|_{L^{2}(\partial \Omega)}^2}
    \leq C \frac{ \|\varphi\|_{H^{-1/2}(\partial \Omega)}^2 }{  \|\varphi\|_{H^{-1}(\partial \Omega)}^2   },
\end{equation}
where $C>0$ only depends on $n, r_0, M_0, |\Omega|, \alpha_0,
\gamma_0, M$.
\end{proof}

\medskip
{\bf Acknowledgment.} The authors wish to thank Ching-Lung Lin,
who kindly pointed out the quantitative estimates of strong unique
continuation for the power of the Laplacian obtained in reference
\cite{l:lnw08}.

%
%
\bibliographystyle{alpha}

\end{document}